\documentclass[twoside,11pt]{amsart}

\usepackage[english]{babel}
\usepackage[utf8x]{inputenc}
\usepackage{url}
\usepackage{amsmath}
\usepackage{graphicx}
\usepackage[colorinlistoftodos]{todonotes}
\usepackage{amsmath,amscd,amsthm}
\usepackage{amstext, amsfonts, a4}
\usepackage{amssymb}
\usepackage{tikz-cd}
\usepackage{fancyhdr}
\usepackage{enumerate}
\usepackage{cleveref}
\usepackage{xcolor}

\newtheorem{thmI}{Theorem}

\numberwithin{equation}{section}
\newtheorem{thm}[equation]{Theorem}

\newtheorem{prop}[equation]{Proposition}
\newtheorem{cor}[equation]{Corollary}
\newtheorem{lem}[equation]{Lemma}

\newtheorem{conj}[equation]{Conjecture}
\theoremstyle{definition}

\newtheorem{ex}[equation]{Example}

\newtheorem{rem}[equation]{Remark}

\renewcommand{\dim}{\operatorname{\mathsf{dim}}}

\renewcommand{\bmod}{\operatorname{\,\,\mathsf{mod}}\,\,}
\newcommand\ind{\operatorname{\mathsf{ind}}}
\renewcommand\exp{\operatorname{\mathsf{exp}}}

\newcommand\hh{\operatorname{\mathbb{H}}}
\newcommand\N{\operatorname{N}}
\newcommand{\car}{\mathsf{char}}

\newcommand{\vf}{\varphi}
\newcommand{\mg}[1]{{#1}^{\times}}
\newcommand{\sq}[1]{{#1}^{\times 2}}
\newcommand{\scg}[1]{\mg{#1}/\sq{#1}}
\newcommand{\s}{\sigma}
\newcommand{\nat}{\mathbb{N}}

\newcommand{\la}{\langle}
\newcommand{\ra}{\rangle}
\newcommand{\lla}{\la\!\la}
\newcommand{\rra}{\ra\!\ra}      

\newcommand{\an}{\mathsf{an}}

\renewcommand{\leq}{\leqslant}
\renewcommand{\geq}{\geqslant}

\newcommand{\I}{\mathsf{I}}

\newcommand{\mc}{\mathcal}
\renewcommand{\N}{\mathsf{N}}

\newcommand{\disc}{\mathsf{disc}}

\newcommand{\G}{\mathsf{G}}
\renewcommand{\H}{\mathsf{H}}

\newcommand{\Hyp}{\mathsf{Hyp}}

\renewcommand{\setminus}{\smallsetminus}

\renewcommand{\dim}{\mathsf{dim}}
\renewcommand{\leq}{\leqslant}
\renewcommand{\geq}{\geqslant}

\newcommand{\Sp}{\mathsf{Sp}}
\newcommand{\wi}{\mathsf{i}}
\newcommand{\mf}{\mathfrak}
\newcommand{\mfm}{\mf{m}}
\newcommand{\zz}{\mathbb{Z}}
\newcommand{\cc}{\mathbb{C}}

\newcommand{\sign}{\mathsf{sign}}

\renewcommand{\bmod}{\,\mathsf{mod}\,}

\renewcommand{\setminus}{\smallsetminus}

\makeatletter
\newcommand{\bigperp}{%
  \mathop{\mathpalette\bigp@rp\relax}%
  \displaylimits
}

\newcommand{\bigp@rp}[2]{%
  \vcenter{
    \m@th\hbox{\scalebox{\ifx#1\displaystyle2.1\else1.5\fi}{$#1\perp$}}
  }%
}
\makeatother

\title[Rational connectedness for proper projective similitudes]{Rational connectedness for\\ groups of proper projective similitudes}

\date{26 June, 2025}

\author{M.~Archita}
\author{Karim Johannes Becher}

\address{University of Antwerp, Department of Mathematics, Antwerp, Belgium.}
\email{karimjohannes.becher@uantwerpen.be}
\email{archita.mondal@uantwerpen.be}

\begin{document}

\maketitle

\begin{abstract}
    For a quadratic form $\varphi$ over a field of characteristic different from $2$, we study whether its group of proper projective similitudes ${\bf PSim}^+(\vf)$ is rationally connected (i.e.~$R$-trivial).
    We obtain new sufficient conditions in terms of structure properties of $\varphi$. We further provide new examples of quadratic forms $\varphi$ belonging to a given power of the fundamental ideal in the Witt ring and such that ${\bf PSim}^+(\vf)$ is not rationally connected. 


\medskip\noindent
{\sc Classification (MSC 2020):} 11E04, 
11E57, 
11E81, 
14E08, 
20G15 
\end{abstract}

\section{Introduction}

Determining whether a variety is rational is a central problem in algebraic geometry.
In this context, the notion of \emph{rational equivalence} (or \emph{$R$-equivalence}) was introduced by Y.~Manin around 1972 (see \cite[Chapter II, \S$14$]{Man86}). 
A systematic study of this notion for linear algebraic groups was then initiated by J.-L.~Colliot-Thélène and J.-J.~Sansuc \cite{CTS77}.
Around 1996, A.~Merkurjev \cite{Mer96} and Ph.~Gille \cite{Gil97} studied rational connectedness (also called \emph{$R$-triviality}) for the group of proper projective similitudes of a quadratic form or, more generally, an involution on a central simple algebra.
Merkurjev showed that rational connectedness is in this case determined by the quotient of two subgroups of the multiplicative group of the underlying field.
This led to the first examples of adjoint classical groups of type $D$ which are not rational.

In the present article, we investigate  rational connectedness for certain adjoint groups of type $D$, namely for the groups of proper projective similitudes of quadratic forms. 
We use standard notation and terminology in quadratic form theory. 
Typical references are \cite{EKM08} and \cite{Lam05}.

In this article, the term \emph{quadratic form} always refers to a nondegenerate quadratic form.
We use the sign $\simeq$ for isometry.
Let $K$ be a field of characteristic different from $2$.
For $n\in\nat$ and $a_1,\dots,a_n\in\mg{K}$, we denote by $\lla a_1,\dots,a_n\rra$ the $n$-fold Pfister form $\la 1,-a_1\ra\otimes\dots\otimes\la 1,-a_n\ra$ over $K$.
Let $\vf$ be a quadratic form over $K$. We denote by $\dim(\vf)$ its \emph{dimension} (number of variables).
If $n=\dim(\vf)$,
Then $\vf\simeq \la a_1,\dots,a_n\ra$ for some $a_1,\dots,a_n\in\mg{K}$.
We denote by $\disc(K)$ the \emph{discriminant of $\vf$}, given as the class $(-1)^{\binom{n}{2}}a_1\dots a_n\sq{K}$ in the square-class group $\scg{K}$ for an arbitrary diagonalisation $\la a_1,\dots,a_n\ra$ of $\vf$.
We further denote by $\mc{C}(\vf)$ the \emph{Clifford algebra of $\vf$} and by $\mc{C}_0(\vf)$
 its \emph{even part}.
We denote by  ${\bf PSim}(\varphi)$ the group of projective similitudes of $\varphi$, viewed as a linear algebraic group, and by ${\bf PSim}^+(\varphi)$ its connected component of the identity.
The elements of the latter are called \emph{proper projective similitudes of~$\varphi$}.

If $\dim(\vf)$ is odd, then  ${\bf PSim}^+(\vf)$ is rational; see \cite[p.~200]{Mer96}.
Therefore the study focuses on forms of even dimension.
If $\dim(\vf)<6$, then ${\bf PSim}^+(\vf)$ is rational, by \cite[Prop.~5]{Mer96}.
The first interesting cases occur in dimension $6$. 
When $\dim(\vf)=6$, \cite[Theorem 3]{Mer96} gives that
${\bf PSim}^+(\vf)$ is stably rational if and only if it is rationally connected, if and only if $\mc{C}_0(\vf)$ is a division algebra.
Using this, one finds examples of $6$-dimensional quadratic forms $\vf$ over many fields $K$ for which ${\bf PSim}^+(\vf)$ is not rationally connected.

The study of rational connectedness of ${\bf PSim}^+(\varphi)$ relies crucially on Merkurjev's characterization in terms of similarity factors of $\vf$ and norms in finite field extensions where $\varphi$ becomes hyperbolic.
This characterization will be revisited in \Cref{S:Merk}.

As a first application we obtain the following variation to an observation by
Merkurjev's in \cite[Prop.~7]{Mer96}, namely that ${\bf PSim}^+(\vf)$ is stably rational and hence rationally connected if $\varphi$ is a tensor product of a Pfister form with a quadratic form of odd dimension.

\begin{thmI}[\Cref{no-odd-zd-ratcon}]
Let $\psi$ be a quadratic form of odd dimension over $K$.
Then ${\bf PSim}^+(\vf\otimes\psi)$ is rationally connected if and only if ${\bf PSim}^+(\vf)$ is rationally connected.
\end{thmI}

One central aim of this article is to obtain new sufficient conditions for rational connectedness of ${\bf PSim}^+(\vf)$.
In order to state our results, we need to recall a few terms.

We denote by $\vf_\an$ the \emph{anisotropic part} and by $\wi(\vf)$ the Witt index of $\vf$. Hence $\vf\simeq \vf_\an\perp \wi(\vf)\times\hh$ where $\hh$ denotes the hyperbolic plane.
For a field extension $L/K$, we denote by $\vf_L$ the quadratic form obtained from $\vf$ by extending scalars to $L$.
The set
$\Sp(\vf)=\{\wi(\vf_L)\mid L/K\mbox{ field extension}\}\subseteq\nat$
is called the \emph{splitting pattern of $\vf$}, and the natural number 
 $h(\vf)=|\Sp(\vf)|-1$
is called the \emph{height of $\vf$}; see \cite[Sect.~25]{EKM08}.
Note that $h(\vf)\leq \frac{1}2\dim(\vf)$, and further $h(\vf)\leq 1$ if and only if $\vf$ is Witt equivalent to a non-trivial scaled Pfister form; see \cite[Prop.~25.6]{EKM08}.

In \Cref{S:splitting} we establish the following result:

\begin{thmI}[\Cref{h2ratcon}, \Cref{2Pfister-ratcon} \& \Cref{no-odd-zd-ratcon}]\label{T2}
    Assume that either $h(\varphi)\leq 2$ or $\varphi$ is Witt equivalent to a sum of two scaled Pfister forms of different dimension.
Then ${\bf PSim}^+(\vf)$ is rationally connected.
More generally, ${\bf PSim}^+(\vf\otimes\psi)$ is rationally connected for any odd-dimensional quadratic form~$\psi$.
\end{thmI}

 \Cref{T2} encompasses two special cases treated by R.~Parimala, J.-P.~Tignol and R.M.~Weiss in \cite[Theorem 1.1]{PTW12}, namely when $\vf$ has trivial discriminant and either $\dim(\vf)=8$ and $\mc{C}(\vf)$ is equivalent to a quaternion algebra, or  
$\dim(\vf)=12$ and $\mc{C}(\vf)$ is split.
Either of the two implies that $h(\vf)\leq 2$.
When $\dim(\vf)=6$, then it follows by \cite[Theorem 3]{Mer96} that \Cref{T2}
 gives precisely the cases where ${\bf PSim}^+(\vf)$ is rationally connected.
One may ask whether this is equally the case when $\dim(\vf)=8$.

One may speculate whether it is possible to extend the list in \Cref{T2} to a full characterization of rational connectedness of ${\bf PSim}^+(\vf)$ in terms of structure properties of $\vf$.
Note that there is no confirmed case of a quadratic form $\vf$ which is a  tensor product of two even-dimensional forms and for which ${\bf PSim}^+(\vf)$ is not rationally connected.

Let $n\in\nat$. We denote by $\I^nK$ the $n$th power of the fundamental ideal $\I K$ of the Witt ring of $K$.
We say that the quadratic form $\vf$ \emph{belongs to} $\I^nK$ and we write $\vf\in\I^n K$ if its Witt equivalence lies in $\I^n K$.
Note that $\vf\in\I^2K$ if and only if $\vf$ has even dimension and trivial discriminant. 

In \Cref{S:fui-conditions}, extending previous results from \cite{KP08} and \cite{AP22} on the case $n=2$, we show the following.

\begin{thmI}[\Cref{T:In-G-H}]\label{T5}
    Assume that $\I^{n+1} K(\sqrt{-1})=0$.
    Let $\varphi\in\I^nK$ and $a\in\mg{K}$.
    Then $a\varphi\simeq \varphi$ if and only if $a$ is a norm in some finite $2$-extension $L/K$ such that $\varphi_L$ is hyperbolic.
\end{thmI} 

Results of this type shed a light on the limitations to the  role of cohomological invariants in determining rational connectedness.

Our other results concern the construction of new examples where ${\bf PSim}^+(\vf)$ is not rationally connected.
Ph.~Gille \cite{Gil97} produced examples of quadratic forms $\vf$ of dimension $8$ and trivial discriminant such that ${\bf PSim}^+(\varphi)$ not rationally connected. 
A closer look at those examples yields that there $\mc{C}(\vf)$ has index $8$.
The following theorem (see \Cref{8I2-followup}) produces examples of a different type over the field $K=\cc(t_1,t_2,t_3)$.
Since this field has cohomological dimension $3$, this sharpens the main result of \cite{Gil97}.

\begin{thmI}[\Cref{8I2-followup} \& \Cref{Gille-cd3examples-simplified}]\label{T3}
    Let $K$ be the rational function field in two variables $k(t_1,t_2)$ over a field $k$ with $\car(k)\neq 2$ and $|\scg{k}|\geq 8$.
    Then there exist $a_0,a_1,a_2\in \mg{K}$ such that ${\bf PSim}^+(\varphi)$ is not rationally connected for the $8$-dimensional quadratic form $\vf=\lla t_1,t_2\rra\perp -a_0\lla a_1t_1,a_2t_2\rra$ over $K$.
\end{thmI}

For any $n\in\nat$, N.~Bhaskhar \cite{Bha14} obtained examples of fields $K$ with quadratic forms $\vf\in \I^nK$ such that ${\bf PSim}^+(\vf)$ is not rationally connected. 
The construction used in \cite{Bha14} produces such forms of relatively large dimension.
The following result provides such examples in dimension $2\cdot(2^n-1)$, when $n\geq 3$.

\begin{thmI}[\Cref{In-ex} \& \Cref{E:In-PSim-nrc}]\label{T4}
    Let $K=k(t_0,\ldots,t_n)$ for a field $k$ such that $\car(k)\neq 2$ and $|\scg{k}|\geq 4$.
    Then there exist elements $a_1,\ldots,a_n\in\mg{K}$ and a quadratic form $\vf\in\I^nK$ such that  ${\bf PSim}^+(\vf)$ is not rationally connected  and $\vf\perp\hh \simeq \lla a_1t_1,\ldots,a_nt_n\rra\perp -\lla t_1,\ldots,t_n\rra$.
\end{thmI}

In particular for $n=3$, since $12$-dimensional forms in $\I^3K$ are of height at most $2$, we find that $14$ is the smallest dimension where forms $\vf\in\I^3K$
can occur such that ${\bf PSim}^+(\vf)$ is not rationally connected.
Moreover, \Cref{T4} shows that \Cref{T2} cannot be extended to cover sums of two scaled Pfister forms in general without requesting them to be of different dimension.

\Cref{T3} and \Cref{T4} both follow from a general method, due to Ph.~Gille \cite{Gil97}, to obtain 
 quadratic forms $\varphi$ such that ${\bf PSim}^+(\varphi)$ is not rationally connected.
 It combines the construction of forms over iterated power series fields with a construction of triquadratic extensions with interesting norm conditions, due to J.-P.~Tignol \cite{Tig81}.
We combine it here with more recent constructions of such triquadratic extensions provided by A.~Sivatski in \cite{Siv10} and \cite{Siv18}, which we restate in \Cref{S:multiquad}.
In \Cref{S:henselian-stuff}, we give a comprehensive account of this method, and in \Cref{S:construction} we apply it to obtain our new examples.

While our study here is limited to the quadratic form case, it should be extended to cover adjoint classical groups in general,  replacing quadratic forms by algebras with involution.

\section{Merkurjev's criterion}
\label{S:Merk}

We now delve a bit more in the setup developed by Merkurjev in \cite{Mer96} that is at the basis of all results concerning rational connectedness of ${\bf PSim}^+(\vf)$.

For a finite field extension $L/K$, we write 
$\N_{L/K}:L\to K$ for the norm map and we abbreviate $$\N_{L/K}^\ast=\N_{L/K}(\mg{L})\,.$$

For a quadratic form $\varphi$ over $K$, we set
$$\G(\vf)=\{a\in\mg{K}\mid a\vf\simeq \vf\}\,,$$
and for a field extension $K'/K$ we denote by $\varphi_{K'}$ the quadratic form $\varphi$ when considered over the field $K'$.

More generally, we consider $r\in\nat$ and a tuple $\varphi=(\varphi_1,\dots,\varphi_r)$ of quadratic forms over $K$.
We set $\G(\varphi)=\bigcap_{i=1}^r\G(\varphi_i)$.
For a field extension $K'/K$, we set $$\varphi_{K'}=((\varphi_1)_{K'},\dots,(\varphi_r)_{K'})\,.$$

We write $\Hyp(\vf)$ (resp.~$\Hyp_2(\vf)$) for the subgroup of $\mg{K}$ generated by the sets $\N_{L/K}^\ast$ where $L/K$ ranges over all finite field extensions (resp.~all finite $2$-extensions) such that $(\varphi_i)_L$ hyperbolic for $1\leq i\leq r$.
We further set 
$$\H(\vf)=\sq{K}\cdot\Hyp(\vf)\,.$$
Trivially, we have the inclusions $\Hyp_2(\vf)\subseteq\Hyp(\vf)\subseteq\H(\vf)$ and $\sq{K}\subseteq\G(\vf)$.
In view of Scharlau's Norm Principle \cite[Chap.~VII, Sect.~4]{Lam05}, we further have  that $\Hyp(\vf)\subseteq \G(\vf)$ and hence $$\H(\vf)\subseteq \G(\vf)\,.$$

\begin{prop}\label{no-odd-zero-divisors}
Let $\varphi$ and $\psi$ quadratic forms over $K$ such that $\dim(\psi)$ is odd.
Then $\G(\vf\otimes\psi)=\G(\vf)$, $\Hyp(\vf\otimes\psi)=\Hyp(\vf)$ and $\H(\vf\otimes\psi)=\H(\vf)$.
\end{prop}
\begin{proof}
Since $\dim(\psi)$ is odd, it follows by \cite[Chap.~VIII, Cor.~8.5]{Lam05} that,
for any field extension $L/K$ and any quadratic form $\rho$ over $L$,
$\rho\otimes \psi_L$ is hyperbolic if and only if $\rho$ is hyperbolic.
Applying this for arbitrary finite extensions $L/K$ and $\rho=\vf_L$ yields that $\Hyp(\vf\otimes\psi)=\Hyp(\vf)$, whereby $\H(\vf\otimes\psi)=\H(\vf)$.
Applying the same fact for an arbitrary $a\in\mg{K}$ to $L=K$ and $\rho=\lla a\rra\otimes\vf$ yields that $\G(\vf\otimes\psi)=\G(\vf)$.
\end{proof}

\begin{prop}
    Let $r\in\nat$ and $a_1,\dots,a_r\in\mg{K}$ and $M=K(\sqrt{a_1},\dots,\sqrt{a_r})$ be such that $[M:K]=2^r$.
    Let $\varphi=(\lla a_1\rra,\dots,\lla a_r\rra)$.
    Then we have:
    \begin{eqnarray*}
    \G(\varphi) & = & \quad \bigcap_{i=1}^r\N_{K(\sqrt{\vphantom{b}a_i})/K}^\ast\\
    \Hyp(\varphi) & = & \Hyp_2(\varphi) \,= \,\N_{M/K}^\ast\\
    \H(\varphi) & = & \sq{K}\cdot \N_{M/K}^\ast
    \end{eqnarray*}
\end{prop}
\begin{proof}
We have $\G(\lla a_i\rra)=\N_{K(\sqrt{\vphantom{b}a_i})/K}^\ast$.
This yields the first equality.
For any field extension $K'/K$, the forms $\lla a_i\rra_{K'}$ are hyperbolic for $1\leq i\leq r$ if and only if $a_1,\dots,a_r\in\sq{K'}$, which is if and only if $M\subseteq K'$. This implies the remaining equalities in the statement.
\end{proof}

Let $\vf$ be a quadratic form over $K$.
Its group of proper projective similitudes ${\bf PSim}^+(\vf)$ is an example of a connected semisimple adjoint classical linear algebraic group.
Whether this group is rationally connected is characterised in \cite{Mer96} as follows.

\begin{thm}[Merkurjev]
\label{Mer:T1}
The group ${\bf PSim}^+(\vf)$ is rationally connected if and only if $\G(\vf_{K'})=\H(\vf_{K'})$ holds for every field extension $K'/K$.
\end{thm}

\begin{proof}
    See \cite[Theorem 1]{Mer96}.
\end{proof}

We obtain the following consequence.

\begin{cor}\label{no-odd-zd-ratcon}
Let $\varphi$ and $\psi$ quadratic forms over $K$ such that $\dim(\psi)$ is odd.
Then ${\bf PSim}^+(\vf\otimes \psi)$ is rationally connected if and only if ${\bf PSim}^+(\vf)$ is rationally connected.
\end{cor}
\begin{proof}
Consider an arbitrary field extension $K'/K$.
By \Cref{no-odd-zero-divisors}, we have
$\G(\vf\otimes\psi)_K')=\G(\vf_{K'})$
and $\H((\vf\otimes\psi)_{K'})=\H(\vf_{K'})$.
In particular the equality $\G((\vf\otimes\psi)_{K'})=\H((\vf\otimes\psi)_{K'})$ holds if and only if $\G(\vf_{K'})=\H(\vf_{K'})$.
Having this equivalence for all field extensions $K'/K$, the statement follows by \Cref{Mer:T1}.
\end{proof}

\begin{rem}
        Merkurjev's statement \cite[Theorem 1]{Mer96} characterizes more generally rational connectedness for ${\bf PSim}^+(A,\s)$ for any central simple algebra with involution $(A,\s)$.
    Also the fact used in \Cref{no-odd-zero-divisors} about zero divisors in the Witt ring $\mathsf{W}K$ can be extended to that setting. The annihilator in $\mathsf{W}K$  of any element of the Witt group of $(A,\s)$ seen as a $\mathsf{W}K$-module is contained in the fundamental ideal $\I K$.  
    This can be seen, assuming in the unitary case that $\exp A\leq 2$, by a generic splitting argument based on the results of Karpenko and Tignol in \cite{Kar10}.
    Using this, one can extend \Cref{no-odd-zd-ratcon} to the case where $A$ is not split.
    Since here we bound ourselves to the quadratic form case, we will include details in another article.
\end{rem}

\begin{rem}
         Under the hypotheses of \Cref{no-odd-zd-ratcon}, one obtains using \cite[Prop.~4.5]{Mer98} that ${\bf PSim}^+(\vf\otimes \psi)\times \mathbb{A}^m$ and ${\bf PSim}^+(\vf)\times \mathbb{A}^{m'}$ are birationally isomorphic for some positive integers $m$ and $m'$.
     This yields an alternative proof of \Cref{no-odd-zd-ratcon}. 
\end{rem}

\section{Sufficient conditions for rational conectedness}
\label{S:splitting}

A quadratic form $\vf$ over $K$ is called \emph{split} if $\dim(\vf_\an)\leq 1$. For even-dimensional quadratic forms, \emph{split} is synonymous with \emph{hyperbolic}.

\begin{lem}\label{L:beta}
Let $\vf$ be a quadratic form over $K$ and $a\in\G(\vf)$.
Then  $\vf$ is split or there exists a quadratic field extension $L/K$ with $\wi(\vf_L)>\wi(\vf)$ and $a\in\N_{L/K}^\ast$.
\end{lem}
\begin{proof}
Suppose that $\vf$ is not split. Since $a\in\G(\vf)=\G(\vf_\an)$, it follows in particular that $\lla a\rra\otimes\vf_\an$ is isotropic.
Hence there exists a binary subform $\beta$ of $\vf_\an$ such that $\lla a\rra\otimes \beta$ is isotropic.
Let $d=\disc(\beta)$. Note that $\beta$ is anisotropic, and hence $d\notin\sq{K}$. 
It further follows that $\lla a,d\rra$ is hyperbolic.
Set $L=K(\sqrt{d})$. Then $L/K$ is a quadratic field extension and $a\in\N_{L/K}^\ast$. 
Furthermore, $(\vf_\an)_L$ is isotropic, whereby $\wi(\vf_L)>\wi(\vf)$.
\end{proof}

\begin{prop}\label{P:pfiratcon}
Let $\pi$ be a Pfister form over $K$.
For every $a\in\G(\pi)\setminus\sq{K}$, there exists a field extension $L/K$ with $[L:K]\leq 2$ such that $\pi_L$ is hyperbolic and $a\in\N^\ast_{L/K}$. In particular $\G(\pi)=\Hyp_2(\pi)=\H(\pi)$.
\end{prop}
\begin{proof}
Consider $a\in\G(\pi)\setminus\sq{K}$.
Hence $a\in\G(\pi)$ and as $a\notin\sq{K}$, we have $\dim(\pi)>1$.
If $\pi$ is hyperbolic, then we can choose $L=K$.
Assume that $\pi$ is not hyperbolic.
By \Cref{L:beta}, we can find a quadratic field extension $L/K$ such that $\pi_L$ is isotropic and $a\in\N_{L/K}^\ast$.
Since $\pi$ is a Pfister form, it follows that $\pi_L$ is hyperbolic. 
This argument shows in particular that $\G(\pi)\subseteq\Hyp_2(\pi)$.
Since $\Hyp_2(\pi)\subseteq\H(\pi)\subseteq\G(\pi)$, this completes the proof.
\end{proof}

A field extension $L/K$ is called \emph{biquadratic} if $[L:K]=4$ and $L=K(\sqrt{a},\sqrt{b})$ for certain $a,b\in\mg{K}$.
The following well-known statement is called in \cite{Bha14} the \emph{biquadratic norm trick}.

\begin{lem}\label{biquatrick}
    Let $L_1/K$ and $L_2/K$ two linearly disjoint quadratic extensions of $K$ and let $L=L_1\otimes_KL_2$.
    Then $\sq{K}\cdot\N_{L/K}^\ast= \N_{L_1/K}^\ast\cap\N_{L_2/K}^\ast$.
\end{lem}
\begin{proof}
See e.g.~\cite[Lemma 2.1]{PTW12} for a proof.
\end{proof}

\begin{lem}\label{L:AP-trick}
Let $m\in\nat$. Let $\pi$ and $\rho$ be quadratic forms over $K$ such that 
$\dim(\pi)<2^{m+1}$ and $\rho\in \I^{m+1}K$.
Then $\G(\pi\perp\rho)=\G(\pi)\cap\G(\rho)$.
\end{lem}
\begin{proof}
Clearly, $\G(\pi)\cap\G(\rho)\subseteq \G(\pi\perp\rho)$.
To show the converse inclusion, consider $a\in\G(\pi\perp\rho)$ arbitrary.
Then $\lla a\rra\otimes \pi \perp\lla a\rra\otimes \rho$ is hyperbolic.
Hence $\lla a\rra\otimes\pi\in I^{m+2}K$.
Since $\dim(\lla a\rra\otimes\pi)<2^{m+2}$, we conclude from the Arason-Pfister Hauptsatz \cite[Theorem 23.7]{EKM08} that $\lla a\rra\otimes\pi$ is hyperbolic, whereby $a\in\G(\pi)$. As $a\in\G(\pi\perp\rho)$, we conclude that $a\in\G(\rho)$, whereby $a\in\G(\pi)\cap\G(\rho)$.
\end{proof}
 
\begin{thm}\label{2Pfister-ratcon}
Let $\vf$ be a quadratic form over $K$ which is Witt equivalent to a sum of two scaled Pfister forms over $K$ of different dimension.
Then $\G(\vf)=\H(\vf)$ and ${\bf PSim}^+(\varphi)$ is rationally connected.
\end{thm}
\begin{proof}
Let $m,n\in\nat$ be such that $\vf$ is Witt equivalent to $\pi\perp\rho$ for an $m$-fold Pfister form $\pi$ and an $n$-fold Pfister form $\rho$ over $K$.
It follows by \Cref{L:AP-trick} that $\G(\vf)=\G(\pi)\cap\G(\rho)$.
Note that $ \sq{K}\Hyp_2(\vf)\subseteq\H(\vf)\subseteq \G(\vf)$. To show the opposite inclusion,
consider an element $a\in\G(\vf)\setminus\sq{K}$. 
Then $a\in\G(\pi)\cap\G(\rho)$.
Since $\pi$ and $\rho$ are scaled Pfister forms, it follows by \Cref{P:pfiratcon} that there exist field extensions $L/K$ and $L'/K$
of degree at most $2$ such that $\pi_L$ and $\rho_{L'}$ are hyperbolic and $a\in\N_{L/K}^\ast\cap\N_{L'/K}^\ast$.
In particular, in view of \Cref{biquatrick}  there exists a multiquadratic field extension $M/K$ with $[M:K]\leq 4$ such that $\vf_M$ is hyperbolic and $a\in\sq{K}\cdot\N^\ast_{M/K}$.
This  shows that $\G(\varphi)\subseteq\sq{K}\Hyp_2(\vf)$.
Hence we have $\G(\vf)=\sq{K}\Hyp_2(\vf)=\H(\vf)$.

Note that the hypothesis on $\vf$ is stable under field extensions. Hence we obtain also for every field extension $K'/K$ that $\G(\varphi_{K'})=\H(\varphi_{K'})$.
By \Cref{Mer:T1}, this shows that ${\bf PSim}^+(\vf)$ is rationally connected.
\end{proof}

\begin{thm}\label{h2ratcon}
    Let $\vf$ be a quadratic form over $K$ such that $h(\vf)\leq 2$.
    Then $\G(\vf)=\H(\vf)=\sq{K}\Hyp_2(\vf)$ and ${\rm\bf PSim}^+(\vf)$ is rationally connected.
\end{thm}
\begin{proof}
    Since $\sq{K}\Hyp_2(\vf)\subseteq\H(\vf)\subseteq \G(\vf)$ holds in general, we need to show that $\G(\varphi)\subseteq\sq{K}\Hyp_2(\vf)$.
    We may assume that $\dim(\vf)$ is even. 
    If $\vf$ is hyperbolic, then $\G(\vf)=\mg{K}=\Hyp_2(\vf)$. Assume that $\vf$ is not hyperbolic.
    Consider an element $a\in\G(\vf)$. 
    By \Cref{L:beta}, there exists a quadratic field extension $L/K$ such that  $\wi(\vf_L)>\wi(\vf)$ and $a\in\N_{L/K}^\ast$.
    Let $d\in\mg{K}\setminus\sq{K}$ be such that $L=K(\sqrt{d})$.
    By \cite[Cor.~22.12]{EKM08}, we can write $\vf_\an\simeq \lla d\rra\otimes \vartheta\perp\psi$ for some quadratic forms $\vartheta$ and $\psi$ over $K$ such that $\psi_L$ is anisotropic. Then $(\vf_L)_\an=\psi_L$.
    We have that $a\in\G(\vf)\cap\G(\lla d\rra)\subseteq\G(\psi)$ and $\psi$ is even-dimensional.
    If $\psi$ is trivial ($0$-dimensional), then $\vf_L$ is hyperbolic, whereby $a\in\H(\vf)$.
    Assume that $\psi$ is not trivial.
    By \Cref{L:beta}, there exists a quadratic extension $L'/K$ such that $\psi_{L'}$ is isotropic and $a\in\N_{L'/K}^\ast$.
    Since $\psi_{L}$ is anisotropic, it follows that $L/K$ and $L'/K$ are linearly disjoint.
    We set $M=L\otimes_KL'$.
    By \Cref{biquatrick}, we have that $\N_{L/K}^\ast\cap\N_{L'/K}^\ast=\sq{K}\cdot\N_{M/K}^\ast$.
    Since $\psi_{L'}$ is isotropic, so is $\psi_M$.
    Since $\psi_L=(\vf_L)_\an$, we obtain that $\wi(\vf_M)>\wi(\vf_L)>\wi(\vf)$.
    As $h(\psi)\leq 2$, it follows that $\vf_M$ is hyperbolic.
    Hence 
    $a\in\N_{L/K}^\ast\cap\N_{L'/K}=\sq{K}\cdot\N_{M/K}^\ast\subseteq\sq{K}\Hyp_2(\vf)$.
    This argument shows that $\G(\varphi)\subseteq \sq{K}\Hyp_2(\vf)$.
Therefore we have the equalities $\G(\varphi)=\sq{K}\Hyp_2(\vf)=\H(\varphi)$.

For any field extension $K'/K$, using that $h(\varphi_{K'})\leq h(\vf)\leq 2$, we obtain that $\G(\varphi_{K'})=\H(\varphi_{K'})$. 
Having this for all field extensions, we conclude by \Cref{Mer:T1} that ${\bf PSim}^+(\vf)$ is rationally connected.
\end{proof}

\begin{thm}[Karpenko]\label{T:Karpenko}
    Let $n\in\nat$ and consider a quadratic form $\varphi\in\I^nK$.
    Then 
    $$\dim(\varphi)\geq (2^{h(\varphi)}-1)\cdot (2^{n+1-h(\varphi)})\,.$$
\end{thm}
\begin{proof}
    This follows from \cite[Cor.~6.2]{Kar04}.
\end{proof}

\begin{cor}
    Let $n\in\nat$ and let $\varphi$ be a quadratic form in $\I^nK$ such that $\dim(\varphi)<7\cdot 2^{n-2}$. Then
      ${\bf PSim}^+(\varphi)$ is rationally connected.
\end{cor}
\begin{proof}
    This is immediate from \Cref{T:Karpenko} and \Cref{h2ratcon}.
\end{proof}

\begin{conj}
    For every integer $n\geq 3$, there exists a field $L$ such that $\car(L)\neq 2$ and a quadratic form $\varphi\in \I^nL$ with $\dim(\varphi)=7\cdot 2^{n-2}$ such that ${\bf PSim}^+(\varphi)$ is not rationally connected.
\end{conj}

For $n=3$, we will confirm this conjecture in \Cref{14I3ex}.

\section{Conditions on powers of the fundamental ideal}
\label{S:fui-conditions}

In this section, we extend the results from \cite[Sect.~4]{AP22}  concerning quadratic forms.
There, fields of virtual cohomological dimension $2$ are treated.
This condition implies that $\I^3K(\sqrt{-1})=0$, so \cite[Cor.~4.3 \& Theorem~4.4]{AP22} are recovered for $n=2$ from the two results below, respectively.

\begin{prop}\label{P:virt-In-torsion-split}
    Let $n\in\nat\setminus\{0\}$ be such that $\I^{n+1} K(\sqrt{-1})=0$.
    Let $\varphi$ be torsion form  over $K$ with $\varphi\in \I^nK$.
    Then for every $a\in\mg{K}$, there there exists a finite $2$-extension $L/K$ such that $\varphi_{L}$ is hyperbolic and $a\in\N_{L/K}^\ast$.
    In particular $\Hyp_2(\vf)=\mg{K}$.
\end{prop}
\begin{proof}
Consider $a\in\mg{K}$.
We prove by induction on $\dim(\vf_\an)$ that $a\in\N_{L/K}^\ast$ for some finite $2$-extension $L/K$, whereby $a\in\Hyp(\vf)$.
Note that $\dim(\vf)$ is even, because $\vf\in\I^{n}K\subseteq \I K$.
If $\dim(\vf_\an)=0$, then $\vf$ is hyperbolic, and the claim holds with $L=K$.
Assume now that $\dim(\vf_\an)>0$.
As $\I^{n+1}K(\sqrt{-1})=0$, we have by \cite[Cor.~35.27]{EKM08} that $\I^{n+1}K$ is torsion-free.
Since $\vf\in\I^nK$ and $\vf$ is torsion, $\lla a\rra\otimes \vf$ is hyperbolic, whereby $a\in\G(\vf)$.
By \Cref{L:beta}, there exists a quadratic field extension $K'/K$ and some $a'\in\mg{K'}$ such that $\wi(\vf_{K'})>\wi(\vf)$ and $a=\N_{K'/K}(a')$. Hence $\dim((\vf_{K'})_\an)<\dim(\vf_\an)$, and we have $\vf_{K'}\in\I^nK'$.
By \cite[Cor.~35.8]{EKM08}, since $\I^{n+1}K(\sqrt{-1})=0$, we have that $\I^{n+1}K'(\sqrt{-1})=0$.
Therefore, by the induction hypothesis, there exists a finite $2$-extension $L/K'$ such that $(\vf_{K'})_L$ is hyperbolic and $a'\in\N_{L/K'}^\ast$. Then $\vf_L$ is hyperbolic, $L/K$ is a finite $2$-extension and $a=\N_{K'/K}(a')\in\N_{K'/K}(\N_{L/K'}^\ast)=\N_{L/K}^\ast$.
\end{proof}

\begin{thm}\label{T:In-G-H}
Assume that $\I^{n+1} K(\sqrt{-1})=0$.
Let $\varphi$ be a quadratic form over $K$. 
If $\varphi\equiv\pi\bmod \I^nK$ for a scaled Pfister form $\pi$ over $K$, then
$\G(\varphi)=\H(\varphi)$. 
\end{thm}
\begin{proof}
If $\dim(\vf)$ is odd, then $\G(\vf)=\sq{K}=\H(\vf)$.
Assume now that $\dim(\vf)$ is even and that, for some $r\in\nat$, there is a scaled $r$-fold Pfister form $\pi$ over $K$ such that
$\vf\equiv \pi \bmod\I^nK$.
We can make this choice in such way that either $r<n$ or $\pi$ is hyperbolic.
Then $\G(\vf)\subseteq\G(\pi)$, by \Cref{L:AP-trick}.

Since $\H(\vf)\subseteq\G(\vf)$ holds unconditionally, it suffices to show the opposite inclusion.
Consider $a\in\G(\vf)$.
Hence $\lla a\rra\otimes \vf$ is hyperbolic, and it follows in particular that $\sign_P(\vf)=0$ for every ordering $P$ of $K$ containing $-a$.
Let $K'=K(\sqrt{-a})$. 
Note that $a\in\N^\ast_{K'/K}$. 
Since $-a\in\sq{K'}$, it follows that $\sign_{P'}\vf_{K'}=0$ for every ordering $P'$ of $K'$.
We conclude by Pfister's Local-Global Principle \cite[Chap.~VIII, Theorem 3.2]{Lam05} that $\vf_{K'}$ is a torsion form over $K'$.
Since $\pi$ is a scaled Pfister form and $a\in\G(\vf)\subseteq\G(\pi)$, we obtain by \Cref{P:pfiratcon} that there exists an extension $K''/K$ with $[K'':K]\leq 2$ such that $\pi_{K''}$ is hyperbolic and $a\in\N_{K''/K}^\ast$.
Let $L=K''(\sqrt{-a})$. 
By \Cref{biquatrick}, since $a\in\N^\ast_{K'/K}\cap \N^\ast_{K''/K}$, there exists $a'\in\mg{L}$ such that $a\in\sq{K}\N_{L/K}(a')$.
Furthermore, $\vf_{L}$ is a torsion form and $\pi_{L}$ is hyperbolic, whereby $\vf_{L}\in\I^{n}L$.
Since $\I^{n+1}K(\sqrt{-1})=0$, it follows by \cite[Cor.~35.8]{EKM08} (applied twice) that $\I^{n+1}L(\sqrt{-1})=0$.
Since $\vf_{L}\in\I^{n}L$, \Cref{P:virt-In-torsion-split} yields that $a'\in\N_{L'/L}^\ast$ for a finite $2$-extension $L'/L$ such that $\vf_{L'}$ is hyperbolic.
We conclude that $a\in \sq{K}\N_{L'/K}^\ast\subseteq\H(\vf)$.
\end{proof}

\begin{rem}
Note that every quadratic form over $K$ is congruent to a scaled Pfister form modulo $\I^2K$.
So for $n=2$, \Cref{T:In-G-H} says that, if $\I^3K(\sqrt{-1})=0$, then $\G(\vf)=\H(\vf)$ for every quadratic form $\vf$ over $K$.
\end{rem}

\section{Multiquadratic field extensions} 
\label{S:multiquad}

Let $L/K$ be a finite multiquadratic field extension of $K$, that is,
a finite Galois extension with an elementary $2$-abelian Galois group.
Let $\widetilde{\N}^\ast_{L/K}$ denote the intersection of the subgroups $\N^\ast_{K'/K}$ where $K'/K$ ranges over quadratic subextensions of $L/K$.
In the context of the study of $2$-torsion elements in the Brauer group, Tignol and Gille considered the quotient group
$$\Lambda(L/K)=\widetilde{\N}^\ast_{L/K}/\sq{K}\N_{L/K}^\ast\,.$$

For a set $S$ we denote by $\mc{P}(S)$ the power set of $S$.
Given $r\in\nat$, an $r$-tuple $(a_1,\dots,a_r)\in K^r$ and a subset $I\subseteq \{1,\dots,r\}$, we set $$a_I=\prod_{i\in I} a_i\,.$$ 

\begin{prop}\label{Tignol-group-Gille-translation}
Let $r\in\nat$, $a_1,\dots,a_r\in\mg{K}$ and $L=K(\sqrt{a_1},\dots,\sqrt{a_r})$.
Let $\varphi=(\lla a_1\rra,\dots,\lla a_r\rra)$.
Then $\widetilde{\N}^\ast_{L/K}=\G(\varphi)$,
$\N^\ast_{L/K}=\Hyp(\varphi)=\Hyp_2(\varphi)$,
$\sq{K}\N^\ast_{L/K}=\H(\varphi)$ 
and $$\Lambda(L/K)=\G(\varphi)/\H(\varphi)\,.$$
\end{prop}
\begin{proof}
    Let $\mc{P}=\mc{P}(\{1,\dots,r\})$.
     The quadratic subextensions of $L/K$ are given by $K(\sqrt{a_{I}})/K$ with $I\in\mc{P}$.
    For $I\in\mc{P}$, we have that $\N^\ast_{K(\sqrt{a_{I}})/K}=\G(\lla a_{I}\rra)$ and $\bigcap_{i\in I}\G(\lla a_i\rra)\subseteq \G(\lla a_{I}\rra)=\N^\ast_{K(\sqrt{a_{I}})/K}$.
    Hence $\widetilde{N}^\ast_{L/K}=\bigcap_{i=1}^r\G(\lla a_i\rra)=\G(\varphi)$.
    Note that $\lla a_{I}\rra_L$ is hyperbolic for any $I\in\mc{P}$.
    On the other hand, for any field extension $N/K$, we have 
    that $\lla a_1\rra_N,\dots,\lla a_r\rra_N$ are hyperbolic if and only if $a_1,\dots,a_r\in\sq{N}$, if and only if $L\subseteq N$.
    Hence $\Hyp(\varphi)=\Hyp_2(\varphi)=\N^\ast_{L/K}$.
    Now the other equalities now follow from the definitions of $\H(\varphi)$ and $\Lambda(L/K)$.
\end{proof}

A quadratic  extension is a multiquadratic extension of degree $2$. 
Note that a biquadratic extension is a multiquadratic extension of degree $4$.
A multiquadratic extension of degree $8$ is called a \emph{triquadratic extension}.
Trivially we have $\Lambda(L/K)=0$ for any quadratic field extension $L/K$, and it follows from \Cref{biquatrick} that the same is true for every biquadratic field extension $L/K$.
The smallest degree where $\Lambda(L/K)\neq 0$ can occur for a multiquadratic field extension $L/K$ is $8$, that is, for a triquadratic extension.
The first examples were constructed by J.-P.~Tignol \cite{Tig81}.
A fairly general construction of multiquadratic extensions $L/K$ with $[L:K]=2^r$ for any $r\geq 3$ with $\Lambda(L/K)\neq 0$ can be derived from a result by A.~Sivatski \cite{Siv10} concerning the case $r=3$.

\begin{prop}[Sivatski]\label{Tignol-Sivatski-multiquad-nontriv}
    Let $K=k(t)$ for a field $k$ with $\car(k)\neq 2$.
    Let $a_1,a_2\in\mg{k}$ be such that $[k(\sqrt{a_1},\sqrt{a_2}):k]=4$ and $\N_{k(\sqrt{a_1})/k}^\ast\cap \N_{k(\sqrt{a_2})/k}^\ast\not\subseteq\sq{k}$. Then for any integer $r\geq 3$, there exist a multiquadratic field extension $L/K$ with $a_1,a_2\in\sq{L}$, $[L:K]=2^r$ and $\Lambda(L/K)\neq 0$.
\end{prop}
\begin{proof}
    We fix $d\in(\N_{k(\sqrt{a_1})/k}^\ast\cap \N_{k(\sqrt{a_2})/k}^\ast)\setminus\sq{k}$ and $a_3=t^2-4d$. 
    Since $[k(\sqrt{a_1},\sqrt{a_2}):k]=4$, we have that $k$ is infinite, so we may fix $c_3,\dots,c_r\in\mg{k}$ with $c_3=1$ such that $c_3^2,\dots,c_r^2$ are distinct.
    We set $a_i=t^2-4dc_i^2$ for $3\leq i\leq r$ and $L=K(\sqrt{a_1},\dots,\sqrt{a_r})$.
    Then $d\in\bigcap_{i=1}^r\N^\ast_{K(\sqrt{a_i})/K}=\widetilde{\N}^\ast_{L/K}$, and 
    since $a_3,\dots,a_r$ are distinct monic irreducible quadratic polynomials in $k[t]$, we 
    obtain that $[L:K]=2^r$.
    By \cite[Cor.~10]{Siv10}, we have that 
    $d\notin \sq{K}\N^\ast_{K(\sqrt{a_1},\sqrt{a_2},\sqrt{a_3})/K}$.
    Hence $d\notin \sq{K}\N^\ast_{L/K}$, and
    we conclude that $\Lambda(L/K)\neq 0$.
\end{proof}

A.~Sivatski also showed that, for every triquadratic extension, the group $\Lambda$ is nontrivial at least in a weak sense, namely after combining the extension with a linearly disjoint extension of the base field.

\begin{thm}[Sivatski]\label{Sivatski-triquatshift}
Let $a_1,a_2,a_3\in\mg{K}$ with $[K(\sqrt{a_1},\sqrt{a_2},\sqrt{a_3}):K]=8$. 
Then there exists a field extension $K'/K$ such that 
$[K'(\sqrt{a_1},\sqrt{a_2},\sqrt{a_3}):K']=8$ and $\Lambda(K'(\sqrt{a_1},\sqrt{a_2},\sqrt{a_3})/K')\neq 0$.
\end{thm}
\begin{proof}
    See \cite[Theorem 5.5]{Siv18}.
\end{proof}

\section{Hyperbolic norm groups over henselian valued fields} 
\label{S:henselian-stuff}

Let $v$ be a valuation on $K$.
We denote by $\mc{O}_v$ the corresponding valuation ring, by $\mfm_v$ the maximal ideal of $\mc{O}_v$, by $Kv$ the residue field $\mc{O}_v/\mfm_v$ and by
$vK$ the value group $v(\mg{K})$, which is an ordered abelian group with additive notation.
We say that $v$ is \emph{henselian} if 
every polynomial $f\in\mc{O}_v[X]$ having a simple root in $Kv$ has a root in $K$, or equivalently, if $v$ extends uniquely to a valuation on any algebraic field extension of $K$.
\medskip

In the sequel, let $v$ be a henselian valuation on $K$ with $\car(Kv)\neq 2$.

\begin{lem}\label{rigid-norms}
    For $t\in\mg{K}$ with $v(t)\notin 2vK$, we have $\N_{K(\sqrt{t})/K}^\ast=\sq{K}\cup t\sq{K}$.
\end{lem}
\begin{proof}
    Since $v$ is henselian and $\car(Kv)\neq 2$, we have $1+\mfm_v\subseteq\sq{K}$.
    Let $t\in\mg{K}$ be such that $v(t)\notin 2vK$.
    For $x,y\in K$ not both zero, we obtain that $v(x^2)=2v(x)\neq v(t)+2v(y)=v(y^2t)$, whereby $v(x^2)<v(ty^2)$ or $v(x^2)>v(ty^2)$ and thus $x^2-ty^2\subseteq \sq{K}(1+\mfm_v)\cup -t\sq{K}(1+\mfm_v)=\sq{K}\cup t\sq{K}$.
    Therefore $\N_{K(\sqrt{t})/K}^\ast=\sq{K}\cup -t\sq{K}$.
\end{proof}

\begin{lem}\label{L:hens-rigid-simfactors}
    Let $\varphi$ be a quadratic form over $K$ and $t\in\G(\varphi)$ be such that 
    $v(t)\notin 2vK$. Then $\varphi_{K(\sqrt{-t})}$ is hyperbolic.
\end{lem}
\begin{proof}  
   Let $K'=K(\sqrt{-t})$ and $K''=K(\sqrt{t})$.
   In view of \Cref{rigid-norms}, we have that 
   $\N_{K'/K}^\ast=\sq{K}\cup t\sq{K}$ and $\N_{K''/K}^\ast=\sq{K}\cup-t\sq{K}$.
    Since $\G(\varphi)=\G(\varphi_\an)$, we may replace $\varphi$ by $\varphi_\an$ and assume that $\varphi$ is anisotropic. 
   Suppose first that $\dim (\varphi)=2$.
   Then $\varphi\simeq \la a,-ad\ra$ for some $a,d\in\mg{K}$.
    Since $\la 1,-t\ra\otimes \varphi$ is hyperbolic and $\varphi$ is anisotropic, we obtain that $d\in\N^\ast_{K''/K}\setminus\sq{K}$. 
    Hence $d\sq{K}=t\sq{K}$, whereby $\varphi_{K'}$ is hyperbolic.  
   In the general case, we obtain by \cite[Cor.~2.3]{EL73} that $\varphi\simeq \beta_1\perp\ldots\perp\beta_r$ for $r=\frac{1}2\dim(\varphi)$ and certain binary quadratic forms $\beta_1,\dots,\beta_r$ over $K$ such that $t\in\G(\beta_i)$ for $1\leq i\leq r$.
   By the $2$-dimensional case we obtain that $(\beta_i)_{K'}$ is hyperbolic for $1\leq i\leq r$, whereby $\varphi_{K'}$ is hyperbolic.
\end{proof}

\begin{thm}\label{T:2henselian-odd-simfactor}
    Let $\varphi$ be a quadratic form over $K$ such that $v(\G(\varphi))\not\subseteq 2vK$.
    Then
    $$\G(\varphi)=\H(\varphi)=\Hyp_2(\varphi)\,.$$
\end{thm}
\begin{proof}
It suffices to show that $\G(\varphi)\subseteq\Hyp_2(\varphi)$.
Consider $a\in \G(\varphi)$ arbitrary.
By the hypothesis, there exists some $t\in\G(\varphi)$ with $v(t)\notin 2vK$.
If $v(a)\notin 2vK$, then \Cref{L:hens-rigid-simfactors} yields that $\varphi_{K(\sqrt{-a})}$ is hyperbolic, whereby $a\in\N_{K(\sqrt{-a})/K}^\ast\subseteq \Hyp_2(\varphi)$.
Assume now that $v(a)\in 2vK$.
Then $t,at\in\G(\varphi)$ and $v(t),v(at)\notin 2vK$. Hence $\varphi_{K(\sqrt{-t})}$ and $\varphi_{K(\sqrt{-at})}$  are both hyperbolic, by \Cref{L:hens-rigid-simfactors}.
Therefore $a=t^{-1}at\in\N_{K'(\sqrt{-t})}^\ast\cdot \N_{K'(\sqrt{-at})}^\ast\subseteq \Hyp_2(\varphi)$.
\end{proof}

\begin{lem}\label{L:even-ramified-norms}
    Let $L/K$ a finite field extension.
    Let $v'$ be the unique valuation extension of $v$ to $L$.
    Assume that $[v'L:vK]$ is even. 
    Then 
    $\N_{L/K}^\ast\subseteq\sq{K}\cup t\sq{K}$ for some  $t\in \mg{K}$ with $v(t)\notin 2vK$.
\end{lem}
\begin{proof}
We may assume without loss of generality that $L/K$ is separable.
Let $N/K$ be the normal closure of $L/K$ and let $M$ be the fixed field of a $2$-Sylow subgroup of the Galois group of $N/L$. 
Then $M/L$ is a finite extension of odd degree and $N/M$ is a $2$-extension. 
It follows that $\mg{L}\subseteq \sq{L}\cdot \N_{M/L}^\ast$.
Let $w$ be the unique valuation extension of $v$ to $M$.
Then $v'=w|_L$.
Let $K'$ be the fixed field of a $2$-Sylow subgroup of the Galois group of $N/K$ that contains the Galois group of $N/M$.
Then $K'/K$ is an extension of odd degree contained in $M/K$ and $M/K'$ is a $2$-extension.
Since $[K':K]$ is odd, so is $[wK':wK]$.
Let $K''/K'$ be a maximal subextension of $M/K'$ with $wK''=wK'$.
    Since $[wL:vK]$ is even, we obtain that $wM\neq wK'=wK''$.
    Hence $M\neq K''$. Since $M/K'$ is a $2$-extension, it follows that $M/K''$ contains a quadratic extension $K'''/K''$.
    Then $K'''=K''(\sqrt{s})$ for some $s\in \mg{K''}$.
    By the choice of $K''$, we obtain that $wK'''\neq wK''$,
    whereby $w(s)\notin 2wK''$.
    Therefore $\N_{K'''/K''}^\ast= \sq{K''}\cup -s\sq{K''}$.
    Let $t=\N_{K''/K}(-s)$.
    Then $\N_{M/K}^\ast\subseteq 
    N_{K'''/K}^\ast=\N_{K''/K}(\N_{K'''/K''}^\ast)\subseteq\sq{K}\cup t\sq{K}$.
    Since $wK''=wK'$, we have that $[wK'':vK]$ divides $[K':K]$ and hence is odd.
    Therefore $v(t)\notin 2vK$.
\end{proof}

Given a finite field extension $L/K$, the valuation $v$ is called \emph{unramified in $L/K$} if for the unique extension $v'$ of $v$ to $L$, we have $v'L=vK$ and the residue field extension $Lv'/Kv$ is separable.

\begin{prop}\label{P:hens-hyp-norms-trichotomy}
    Let $M/K$ be a finite field extension and  $\varphi$ a quadratic form such that $\varphi_M$ is hyperbolic.
    Let $L/K$ be a maximal subextension of $M/K$ in which $v$ is unramified.
    Then one of the following holds:
    \begin{enumerate}[$(i)$]
        \item $\varphi_{L}$ is hyperbolic. 
        \item $\N_{M/K}^\ast\subseteq\sq{K}$.
        \item $\N_{M/K}^\ast\subseteq \sq{K}\cup t\sq{K}=\N^\ast_{K(\sqrt{-t})/K}$ for some $t\in \mg{K}$ with $v(t)\notin 2vK$ and such that $\varphi_{K(\sqrt{-t})}$ is hyperbolic.
    \end{enumerate}
\end{prop}
\begin{proof}
    Assume that $(i)$ and $(ii)$ do not hold.
    Hence $\varphi_L$ is not hyperbolic and there exists $x\in \mg{K}$ such that $\N_{M/K}(x)\notin\sq{K}$. 
    Since $\varphi_M$ is hyperbolic, it follows by \cite[Chap.~VII, Cor.~2.6]{Lam05}
    that $[M:L]$ is even.
    By the choice of $L/K$, this implies that $[wM:wL]$ is even, where $w$ is the unique extension of $v$ to $M$.
    It follows by \Cref{L:even-ramified-norms} that $\N^\ast_{M/K}\subseteq \sq{K}\cup t\sq{K}$ for some $t\in\mg{K}$ with $v(t)\notin 2vK$.
    As $\N_{M/K}(x)\notin \sq{K}$, it follows that $\N_{M/K}(x)\in t\sq{K}$.
    Furthermore,  Scharlau's Norm Principle \cite[Chap.~VII, Theorem~4.3]{Lam05} yields that $\N_{M/K}(x)\in \G(\varphi)$.
    Hence $t\in\G(\varphi)$, and we conclude by \Cref{L:hens-rigid-simfactors} that $\varphi_{K(\sqrt{-t})}$ is hyperbolic.
    \end{proof}

The form $\varphi$ is called \emph{$v$-unimodular} if $\vf\simeq \la a_1,\dots,a_n\ra$ for $n=\dim(\vf)$ and certain $a_1,\dots,a_n\in\mg{\mc{O}}_v$.

\begin{cor}\label{C:equimodular-dec}
    Let $\varphi$ be an anisotropic quadratic form over $K$. 
    Let $r\in\nat$, $\varphi_1,\dots,\varphi_r$\,  $v$-unimodular quadratic forms  over $K$ and $t_1,\dots,t_r\in \mg{K}$ with $v(t_i)\not\equiv v(t_j)\bmod 2vK$ for $1\leq i<j\leq r$ such that $\varphi\simeq  t_1\varphi_1\perp\ldots\perp t_r\varphi_r$.
Assume that $v(\G(\varphi))\subseteq 2vK$.
Then $\H(\varphi)=\sq{K}\H(\vf_1,\dots,\vf_r)$. 
\end{cor}
\begin{proof}
Since $v(\G(\varphi))\subseteq 2vK$,
the situation $(iii)$ in \Cref{P:hens-hyp-norms-trichotomy} does not occur. So \Cref{P:hens-hyp-norms-trichotomy} yields that $\H(\varphi)$ is generated by $\sq{K}$ and the groups $\N_{L/K}^\ast$ where $L/K$ is a $v$-unramified finite field extension such that $\varphi_L$ is hyperbolic.

Consider a finite field extension $L/K$ in which $v$ is unramified.
Let $w$ denote the unique valuation extension of $v$ to $L$.  Then $w$ is henselian, and since $L/K$ is $v$-unramified, we have $wL=vK$.
    In particular $w(t_i)\not\equiv w(t_j)\bmod 2wL$ for $1\leq i<j\leq r$.
    Since $\varphi_L\simeq t_1(\varphi_1)_L\perp\ldots\perp t_r(\varphi_r)_L$ and $w$ is henselian, 
    it follows by Durfee's Theorem (cf. \cite[Theorem 3.11]{BDM25}) that $\varphi_L$ is hyperbolic if and only if $(\varphi_1)_L,\dots,(\varphi_r)_L$ are hyperbolic. This together yields the statement.
\end{proof}

\begin{rem}
Based on \Cref{P:hens-hyp-norms-trichotomy}, one can easily retrieve and extend \cite[Prop.~1]{Gil97} as well as \cite[Prop.~5.2 \& 5.3]{Bha14}, while removing the restriction to characteristic $0$ in both cases. 

In \cite[Sect.~5]{Bha14}, it is assumed that 
$\H(\vf)\neq \G(\vf)$.
The latter implies that $\vf$ is not Witt equivalent to any scaled Pfister form.
Hence  \cite[Prop.~5.2 \& 5.3]{Bha14} is covered by the following variation.
\end{rem}

\begin{cor}[Bhaskhar]\label{Bhakhar-S5}
    Let $K=k(\!(t)\!)$ for a field $k$ with \hbox{$\car(k)\neq 2$}.  
    Let $\pi$ be an anisotropic Pfister form over $k$ and $\vf$ a quadratic form over $k$ which is not Witt equivalent to any scalar multiple of $\pi$.
    Consider the quadratic form $\rho=\varphi\perp t\pi$ over $K$.
    Then $\H(\rho)\subseteq \H(\vf)\cdot \sq{K}$. 
\end{cor}
\begin{proof}
    This is immediate from \Cref{P:hens-hyp-norms-trichotomy}.
\end{proof}

\section{Constructing examples over iterated power series fields} 
\label{S:construction}

We will now construct, for some given $n\in\nat$, a quadratic form $\rho\in \I^nK$ for some field $K$ such that ${\bf PSim}^+(\rho)$ is not rationally connected. 
To obtain such examples, we will take an iterated power series field $K=k(\!(t_1)\!)\ldots(\!(t_r)\!)$ in $r$ variables over some base field $k$ and build a form 
$\rho=\bigperp_{I\in\mc{P}} t_{I}(\varphi_I)_K$ over $K$ starting from
a family of $(\varphi_{I})_{\mc{I}\in\mc{P}}$ over $k$, where $\mc{P}$ is the power set of $\{1,\dots,r\}$.
We will use the properties of the sets $\G((\varphi_I)_{I\in\mc{P}})$ and $\H((\varphi_I)_{I\in\mc{P}})$ over $k$ to obtain that $\G(\rho)\neq \H(\rho)$.

An obstruction  is indicated by the following observation.
If we consider a form over the power series field in one variable $K=k(\!(t)\!)$ of the shape $\lla t\rra\otimes \vf_K$ for some quadratic form $\vf$ defined over $k$, then we have the equality $\G(\lla t\rra\otimes \vf_K)=\H(\lla t\rra\otimes \vf_K)$, regardless of the relation between $\G(\vf)$ and $\H(\vf)$, as the next statement indicates.

\begin{prop}
    Assume that $K=k(\!(t)\!)$ for a field $k$ with $\car(k)\neq 2$ and let $\varphi$ be a quadratic form over $k$. 
    Consider the quadratic form $\rho=\varphi\perp t\varphi$ over $K$.
    Then $\H(\varphi_K)=\sq{K}\cdot\H(\varphi)$ and
    $\G(\rho)=\H(\rho)=\Hyp_2(\rho)=\G(\varphi_{K})\cup t\G(\varphi_{K})$. 
    Furthermore, if $\varphi$ is not hyperbolic, then $\G(\varphi_{K})=\sq{K}\cdot\G(\varphi)$.
\end{prop}
\begin{proof}
    Let $v$ denote the $t$-adic valuation on $K=k(\!(t)\!)$. 
    Since $v$ is henselian,  $t\in\G(\rho)$, $vK=\zz$ and $v(t)=1$, \Cref{T:2henselian-odd-simfactor} yields that $\G(\rho)=\H(\rho)=\Hyp_2(\rho)$.

    Given a finite field extension $\ell/k$ such that $\varphi_\ell$ is hyperbolic, we have with $\ell(\!(t)\!)/k$ a finite field extension such that $\varphi_{\ell(\!(t)\!)}$ is hyperbolic and 
    $\N_{\ell/k}^\ast\subseteq\N_{\ell(\!(t)\!)/K}^\ast$.
    This argument shows that $\H(\varphi)\subseteq \H(\varphi_K)$.
    Consider now a finite field extension $L'/K$ such that $\vf_{L'}$ is hyperbolic.
    If $[L':K]$ is odd, then $\varphi_K$ is hyperbolic and consequently $\varphi$ is hyperbolic over $k$.
    Assume now that $[L':K]$ is even.
    Let $\ell$ denote the residue field of the unique  extension of $v$ to $L$.
    Then $\ell/k$ is a finite field extension, $\varphi_\ell$ is hyperbolic, and since $[L':K]$ is even, we  have that
    $\N_{L'/K}^\ast\subseteq\sq{K}\N_{\ell/k}^\ast$.
    This together shows that $\H(\vf_K)= \sq{K}\H(\vf)$.

    Consider $c\in\mg{k}$. We have 
    $\la 1,-c\ra\otimes \rho \simeq (\la 1,-c\ra\otimes \varphi)_K\perp t(\la 1,-c\ra\otimes \varphi)_K$.
    In view of \cite[Chap.~VI, Theorem 1.4]{Lam05}, this form is hyperbolic over $K$ if and only if $\la 1,-c\ra\otimes \varphi$ is hyperbolic over $k$. In other terms, $c\in\G(\rho)$ if and only if $c\in\G(\varphi)$.
    This argument shows that $\G(\rho)\cap\mg{k}=\G(\varphi)$.
    Furthermore $t\in\G(\rho)$.
    Since $\mg{K}=\mg{k}\sq{K}\cup t\mg{k}\sq{K}$, we conclude that $\G(\varphi_K)\cup t\G(\varphi_K)= \G(\rho)$.

    Assume that $\varphi$ is not hyperbolic. It follows by \cite[Chap.~VI, Theorem 1.4]{Lam05} that $v(\G(\varphi_K))\subseteq 2vK$. Hence $\G(\varphi_K)=\G(\rho)\cap \mg{k}\sq{K}=(\G(\rho)\cap\mg{k})\sq{K}=\G(\varphi)\sq{K}$.
\end{proof}

In order to apply \Cref{P:hens-hyp-norms-trichotomy} in the situation of an iterated power series field, we rely on the following well-known fact.
For convenience and by lack of a reference, we include a proof.

\begin{lem}\label{P:itpowser-hensval}
    Assume that $K=k(\!(t_1)\!)\ldots(\!(t_n)\!)$ for some $n\in\nat$ and a field $k$.
    Then $K$ carries a henselian valuation $v$ such that $\mg{\mc{O}}_v=\mg{k}+\mfm_v$ and $vK=(\zz^n,\leq_{\mathsf{lex}_n})$ where $\leq_{\mathsf{lex}_n}$ denotes the lexicographical ordering on $\zz^n$.
\end{lem}
\begin{proof}
    We set $K_0=k$ and $K_r=K_{r-1}(\!(t_r)\!)$ for $1\leq r\leq n$.
    We prove the statement by induction on $n$.
    If $n=0$, then the statement is trivial. 
    Assume that $n>0$. By the induction hypothesis, $K_{n-1}$ carries a henselian valuation $v_{n-1}$ with value group $(\zz^{n-1},\leq_{\mathsf{lex}_{n-1}})$ and such that $\mg{\mc{O}}_{v_{n-1}}=\mg{k}+\mfm_{v_{n-1}}$.
    Let $w$ denote the $t_n$-adic valuation on $K=K_{n-1}(\!(t_n)\!)$.
    Then $\mg{\mc{O}}_w=\mg{K}_{n-1}+\mfm_{v_{n-1}}$ and $w$ is a complete discrete valuation on $K$, hence in particular henselian.
    We take $v$ to be the composition on $w$ with $v_{n-1}$ as defined in \cite[Sect.~2.3]{EP05}.
    Then $vK$ is isomorphic as an ordered abelian group to $wK\times v_{n-1}K_{n-1}$ lexicographically ordered, and hence to $(\zz^n,\leq_{\mathsf{lex}_n})$.
    Since $w$ and $v_{n-1}$ are henselian, so is $v$, by \cite[Cor.~4.1.4]{EP05}.
    Furthermore $\mg{\mc{O}}_v\subseteq \mg{\mc{O}}_{v_{n-1}}+\mfm_w
    \subseteq \mg{k}+\mfm_{v_{n-1}}+\mfm_{w}\subseteq\mg{k}+\mfm_v$, whereby $\mg{\mc{O}}_v=\mg{k}+\mfm_{v}$.
\end{proof}

The following statement captures everything we need to derive new examples from the phenomena for multiquadratic field extensions discussed in \Cref{S:multiquad}.

\begin{prop}\label{GH-glued-families}
    Let $K=k(\!(t_1)\!)\ldots(\!(t_n)\!)$ for a field $k$ with $\car(k)\neq 2$ and some $n\in\nat$.
    Let $r\leq 2^n$, $\varphi_1,\dots,\varphi_r$ quadratic forms over $k$ and $I_1,\dots,I_r$ different subsets of $\{1,\dots,n\}$.
    Over $K$, consider the quadratic form $$\rho=\bigperp_{j=1}^r t_{I_j}(\varphi_j)_K\,.$$
    If $\G(\rho)\subseteq \mg{k}\sq{K}$, then the following hold:
    \begin{eqnarray*}
    \G(\rho) & = & \sq{K}\G(\varphi_1,\ldots,\varphi_r)\\
    \H(\rho) & = & \sq{K} \H(\varphi_1,\ldots,\varphi_r)\\
    \G(\rho)/\H(\rho) & \simeq & \G(\varphi_1,\dots,\varphi_r)/\H(\varphi_1,\dots,\varphi_r)
    \end{eqnarray*}
    If $\G(\rho)\not\subseteq \mg{k}\sq{K}$, then $\G(\rho)=\H(\rho)$ and $\rho_L$ is hyperbolic for a quadratic field extension $L/K$. 
\end{prop}
\begin{proof}
By \Cref{P:itpowser-hensval}, $K$ carries a henselian valuation with $\mg{\mc{O}}_v=\mg{k}+\mfm_v$ and $vK\simeq \zz^n$. 
In particular $\mg{K}\cap v^{-1}(2vK)=\mg{k}\sq{K}$.
If $\G(\rho)\not\subseteq \mg{k}\sq{K}$, then $v(\G(\rho))\notin 2vK$, and it follows by \Cref{L:hens-rigid-simfactors} that $\rho_L$ is hyperbolic for some quadratic field  extension $L/K$ and further by
\Cref{T:2henselian-odd-simfactor} that $\G(\rho)=\H(\rho)$.

Assume now that $\G(\rho)\subseteq \mg{k}\sq{K}$. Then $v(\G(\rho))\subseteq 2vK$.
Since $v$ is henselian and $\mg{\mc{O}}_v=\mg{k}+\mfm_v$, we have 
that $\G(\varphi_K)=\sq{K}\G(\varphi)$ for every non-hyperbolic quadratic form $\varphi$ over $k$.
Since $\rho$ is not hyperbolic, the forms $\varphi_1,\dots,\varphi_r$  over $k$ are not all hyperbolic.
Since $\G(\rho)\subseteq \mg{k}\sq{K}$, it follows that 
$$\G(\rho)=\G((\varphi_1)_K,\dots,(\varphi_r)_K)=\sq{K}\G(\varphi_1,\dots,\varphi_r)\hspace{1cm}\,.$$
It further follows by \Cref{C:equimodular-dec}
that $$\H(\rho)  = \sq{K} \H((\varphi_1)_K,\ldots,(\varphi_r)_K) = \sq{K} \H(\varphi_1,\ldots,\varphi_r)\,.$$
Since $\mg{k}\cap \sq{K}=\sq{k}$, we conclude that 
\begin{eqnarray*}
\hspace{2cm}\G(\rho)/\H(\rho) & \simeq & \G(\varphi_1,\dots,\varphi_r)/\H(\varphi_1,\dots,\varphi_r). \hspace{3cm}\qedhere
\end{eqnarray*}
\end{proof}

In \cite{Gil97}, examples have been constructed of fields allowing for an $8$-dimensional quadratic form $\varphi$ with  trivial discriminant and such that ${\bf PSim}^+(\varphi)$ is not rationally connected.
However, all such examples known to us from the literature are such that $\ind(\mc{C}(\varphi))=8$.
Here, we provide a different construction, which leads to examples where $\vf$ is a sum of two scaled $2$-fold Pfister forms, hence in particular such that
$\ind(\mc{C}(\varphi))=4$.

\begin{thm}\label{8I2index4ex}
    Let $K=k(\!(t_1)\!)(\!(t_2)\!)$ for a field $k$ with $\car(k)\neq 2$ and
     let $a_0,a_1,a_2\in\mg{k}$ be such that $[k(\sqrt{a_0},\sqrt{a_1},\sqrt{a_2}):k]=8$.
     Consider the $8$-dimensional quadratic form 
     $\vf=\lla t_1,t_2\rra\perp -a_0\lla a_1t_1,a_2t_2\rra$ over $K$.
     Then $\G(\vf)/\H(\vf)\simeq \Lambda(k(\sqrt{a_0},\sqrt{a_1},\sqrt{a_2})/k)$ and ${\bf PSim}^+(\vf)$ is not rationally connected.
\end{thm}
\begin{proof}
     Note that $\vf=\lla a_0\rra\perp -t_1\lla a_1\rra\perp -t_2\lla a_2\rra\perp t_1t_2\lla a_0a_1a_2\rra$.
    Since $a_0a_1a_2\in\sq{k(\sqrt{a_0},\sqrt{a_1},\sqrt{a_2})}$, it follows by \Cref{Tignol-group-Gille-translation} and \Cref{GH-glued-families} that $\Lambda(k(\sqrt{a_0},\sqrt{a_1},\sqrt{a_2})/k)\simeq \G(\vf)/\H(\vf)$.
    It follows from \Cref{Sivatski-triquatshift} that there exists a field extension $k'/k$ linearly disjoint to $k(\sqrt{a_0},\sqrt{a_1},\sqrt{a_2})/k$ such that $\Lambda(k'(\sqrt{a_0},\sqrt{a_1},\sqrt{a_2})/k')\neq 0$.
    Replacing $K$ by $K'=k'(\!(t_1)\!)(\!(t_2)\!)$ in the first part, we conclude that $\G(\vf_{K'})\neq \H(\vf_{K'})$.
    Therefore ${\bf PSim}^+(\vf)$ is not rationally connected, in view of \Cref{Mer:T1}.
\end{proof}

\begin{cor}\label{8I2-followup}
    Let $k$ be a field with $\car(k)\neq 2$ and $|\scg{k}|\geq 8$. Then over $k(t_1,t_2)$, there exists an $8$-dimensional quadratic form $\vf$ which decomposes as a sum of two scaled $2$-fold Pfister froms such that ${\bf PSim}^+(\vf)$ is not rationally connected.
\end{cor}
\begin{proof}
       The hypothesis on $k$ means that there exist elements $a_0,a_1,a_2\in \mg{k}$ such that $[k(\sqrt{a_0},\sqrt{a_1},\sqrt{a_2}):k]=8$ holds.
    Consider now the quadratic form 
    $\vf=\lla t_1,t_2\rra\perp -a_0\lla a_1t_1,a_2t_2\rra$ over $k(t_1,t_2)$.
    By extending $\vf$ to $k(\!(t_1)\!)(\!(t_2)\!)$, we conclude from \Cref{8I2index4ex} that ${\bf PSim}^+(\vf)$ is not rationally connected.
\end{proof}

The following example sharpens the main result of \cite{Gil97}.

\begin{ex}\label{Gille-cd3examples-simplified}
    Consider an arbitrary field $k$ with $\car(k)\neq 2$ and the rational function field in three variables $K=k(t_0,t_1,t_2)$.
    Since $|\scg{k(t_0)}|=\infty$, it follows by \Cref{8I2-followup} that  over $K$ there exist $8$-dimensional quadratic forms $\vf$ with trivial discriminant and such that ${\bf PSim}^+(\vf)$ is not rationally connected.
    This holds in particular when $k$ is algebraically closed.
    In that situation, however, we have additionally that 
    $K$ has cohomological dimension $3$ and $u$-invariant $8$.
\end{ex}

\begin{rem}\label{u}
    We refer to \cite[Chap.~8]{Pfi95} for a discussion of the $u$-invariant of a field, and more particularly to \cite[Chap.~8, Ex.~1.2 (6)]{Pfi95} for the fact that, for $n\in\nat$, a rational function field in $n$ variables over an algebraically closed field has $u$-invariant $2^n$.
\end{rem}
    
\begin{thm}\label{In-fork-construction}
    Let 
    $K=k(\!(t_1)\!)\ldots(\!(t_n)\!)$ for a field $k$ with $\car(k)\neq 2$ and some $n\in\nat$.
    Let $a_1,\dots,a_
    n\in\mg{k}$ be such that $[k(\sqrt{a_1},\dots,\sqrt{a_n}):k]=2^n$.
    Let $\mc{P}=\mc{P}(\{1,\dots,n\})\setminus\{\emptyset\}$ and 
    $$\varphi=\bigperp_{I\in \mc{P}}  t_I\lla a_I\rra\,.$$
    Then $\varphi\in\I^nK$, $\dim(\varphi)=2^{n+1}-2$ and $$\G(\varphi)/\H(\varphi)\simeq \Lambda(k(\sqrt{a_1},\dots,\sqrt{a_n})/k)\,.$$
\end{thm}
\begin{proof}
     It is easy to see that $\varphi\perp\hh\simeq \lla -t_1,\dots,-t_n\rra\perp - \lla -a_1t_1,\dots,-a_nt_n\rra$. Hence 
    $\varphi\in\I^nK$ and $\dim(\varphi)=2^{n+1}-2$.
    Since $[k(\sqrt{a_1},\dots,\sqrt{a_n}):k]=2^n$, we have for every $I\in\mc{P}$ that 
    $a_I\notin\sq{k}$, whereby $\lla a_I\rra$ is anisotropic over $k$.
    From this it follows that the quadratic form $\vf$ over $K$ is anisotropic.
    Since $\dim(\varphi)\equiv 2\bmod 4$ and $\varphi\in\I^2K$, it follows that $\varphi_L$ is not hyperbolic for any quadratic field extension $L/K$.
    We conclude by \Cref{Tignol-group-Gille-translation} and \Cref{GH-glued-families} that $\G(\varphi)/\H(\varphi)\simeq \Lambda(k(\sqrt{a_1},\dots,\sqrt{a_n})/k)$.
\end{proof}

\begin{cor}\label{In-ex}
    Let $n\in\nat$ with $n\geq 3$ and $K=k(t_0,\dots,t_n)$ for a field $k$ with $\car(k)\neq 2$ and $|\scg{k}|\geq 4$.
    There exists a quadratic form $\varphi\in\I^nK$ with $\dim(\varphi)=2^{n+1}-2$ and such that ${\bf PSim}^+(\varphi)$ is not rationally connected.
\end{cor}
\begin{proof}
    Set $k'=k(t_0)$. According to \Cref{Tignol-Sivatski-multiquad-nontriv}, there exist $a_1,\dots,a_n\in\mg{k'}$ such that 
    $[k'(\sqrt{a_1},\dots,\sqrt{a_n}):k']=2^n$ and $\Lambda(k'(\sqrt{a_1},\dots,\sqrt{a_n})/k')\neq 0$.
    Let $\mc{P}=\mc{P}(\{1,\dots,n\})\setminus\{\emptyset\}$ and 
    $$\varphi=\bigperp_{I\in \mc{P}}  t_I\lla a_I\rra\,.$$
    Then $\varphi\in\I^nK$, $\dim(\varphi)=2^{n+1}-2$.
    Furthermore, letting $K'=k'(\!(t_0)\!)\ldots(\!(t_n)\!)$, it follows 
    by \Cref{In-fork-construction} that 
    $$\G(\varphi_{K'})/\H(\varphi_{K'})\simeq \Lambda(k'(\sqrt{a_1},\dots,\sqrt{a_n})/k')\neq 0\,.$$
    Hence, according to \Cref{Mer:T1}, ${\bf PSim}^+(\varphi)$ is not rationally connected.
\end{proof}

\begin{ex}\label{E:In-PSim-nrc}
    Taking $k=k_0(t)$ for an algebraically closed field $k_0$ with $\car(k_0)\neq 2$ and a variable $t$, we obtain by \Cref{In-ex} a field $K$ of cohomological dimension $n+2$ with $u$-invariant $2^{n+2}$ and a quadratic form $\varphi\in\I^nK$ of dimension $2^{n+1}-2$ such that ${\bf PSim}^+(\varphi)$ is not rationally connected.
    \end{ex}
    
\begin{cor}\label{14I3ex}
Let $K=k(t_1,t_2,t_3)$ for a field $k$ with $\car(k)\neq 2$ and $|\scg{k}|\geq 4$.
    Then there exists a $14$-dimensional quadratic form $\varphi\in\I^3K$ such that ${\bf PSim}^+(\varphi)$ is not rationally connected.
\end{cor}
\begin{proof}
    By the hypothesis on $k$, there exist elements $a_1,a_2,a_3\in\mg{k}$ such that $[k(\sqrt{a_1},\sqrt{a_2},\sqrt{a_3}):k]=8$.
    By \Cref{Sivatski-triquatshift}, there exists a field extension $k'/k$ such that $[k'(\sqrt{a_1},\sqrt{a_2},\sqrt{a_3}):k']=8$ and $\Lambda(k'(\sqrt{a_1},\sqrt{a_2},\sqrt{a_3})/k')\neq 0$.
    Let $\mc{P}=\mc{P}(\{1,2,3\})\setminus\{\emptyset\}$ and consider the quadratic form 
    $$\varphi=\bigperp_{I\in\mc{P}} t_I\lla a_I\rra$$
    over $K$.
    Let $K^\ast=k'(\!(t_1)\!)(\!(t_2)\!)(\!(t_3)\!)$.
    It follows by \Cref{In-fork-construction} that 
    $$\G(\varphi_{K^\ast})/\H(\varphi_{K^\ast})\simeq \Lambda(k'(\sqrt{a_1},\sqrt{a_2},\sqrt{a_3})/k')\neq 0\,.$$
    Hence, according to \Cref{Mer:T1}, ${\bf PSim}^+(\varphi)$ is not rationally connected.
\end{proof}

\begin{ex}
    Taking $k=k_0(t_0)$ for an algebraically closed field $k_0$ with $\car(k_0)\neq 2$, we obtain from \Cref{14I3ex} an example of a field $K$ of cohomological dimension $4$ with $u$-invariant $16$ admitting a $14$-dimensional quadratic form $\varphi$ in $\I^3K$ such that ${\bf PSim}^+(\varphi)$ is not rationally connected.
\end{ex}

\bibliographystyle{plain}

\end{document}